\documentclass[12pt,a4paper,twoside]{amsart}
\usepackage{a4}
\usepackage{graphicx}
\usepackage{amstext,amsfonts,amsmath,amsthm,graphicx,amssymb,amscd,epsfig,amsbsy}
\newtheorem{theorem}{Theorem}[section]
\newtheorem{lemma}[theorem]{Lemma}
\newtheorem{proposition}[theorem]{Proposition}
\newtheorem{corollary}[theorem]{Corollary}
\theoremstyle{definition}
\newtheorem{definition}[theorem]{Definition}
\newtheorem{example}[theorem]{Example}

\theoremstyle{remark}
\newtheorem{remark}[theorem]{Remark}
\theoremstyle{remark}

\numberwithin{equation}{section}

\begin{document}
\title[Unipotent planar maps]{An spectral condition for global equivalence of planar maps}%
\author[R. Rabanal]{Roland Rabanal}%
\address[R. Rabanal]{Departamento de Ciencias, Pontificia Universidad Cat\'{o}lica del Per\'{u},   Av. Universitaria 1801, Lima~32,  Per\'{u}}
\email{rrabanal@pucp.edu.pe}%
\thanks{The author was partially supported by \textsc{pucp-Peru}~(\textsc{dgi}: 2018-3-003).}%

\subjclass[2010]{Primary 37C15, 37E30; Secondary 26B10, 14R15}%
\date{\today}%
\keywords{Jacobian conjecture, Unipotent maps}%
\begin{abstract}

The spectral condition on the differentiable maps, of the Euclidean plane $\mathbb{R}^2$ into itself, is the assumption that their Jacobian eigenvalues are all equal to one (unipotent maps).
It is demonstrated that \textit{a $C^1-$unipotent map is globally equivalent to the linear translation $\tau(x,y)=(x+1,y)$, as long as the map is fixed point free} (i. e. $G(q)\neq q,\forall q$ implies $\varphi\circ G\circ \varphi^{-1}=\tau$, for some homeomorphism $\varphi\colon\mathbb{R}^2\to\mathbb{R}^2$).
Similarly, it is proved not only that the fixed point set induced by a $C^1-$unipotent has no isolated elements, but that \textit{a $C^1-$unipotent map has no periodic points}.
The relation with the existence of global attractors in $\mathbb{R}^2$, by using a global bifurcation on unipotent maps, is also studied.
\end{abstract}
\maketitle
\section{Introduction}\label{introduction}
A differentiable map $G\colon \mathbb{R}^2\to\mathbb{R}^2$, not necessarily continuously differentiable, is called   \emph{unipotent}, provided that its {\it spectrum}  $\mbox{Spc}(G)=\{\mbox{\rm Eigenvalues of } DG_z: z\in \mathbb{R}^2\}\subset\mathbb{C}$ is the set~$\{1\}$; other words, their Jacobian eigenvalues are all equal to one.
Evidently, the unipotent maps are always orientation preserving maps, as the simplest examples given by the rules $\tau(x,y)=(x+1,y)$ (linear translation) and $\mbox{Id}(x,y)=(x,y)$ (identity map).
In these circumstances, it is acceptable and also reasonable to ask if the identity map and the linear translation describe the global behavior of the unipotent maps, in the sense that a nonlinear unipotent map is \textit{equivalent} to either the linear translation or the identity map.
\par
To describe the precise equivalence under consideration, the global map $G$ is said to be  \emph{globally equivalent} to $F\colon \mathbb{R}^2\to\mathbb{R}^2$, if there exists a homeomorphism $\varphi\colon\mathbb{R}^2\to\mathbb{R}^2$ such that the composition $\varphi\circ G\circ \varphi^{-1}=F$.
This homeomorphism $\varphi$ is denominated a topological conjugacy between $G$ and $F$.
In the particular case that the map $F$ is linear, the conjugacy $\varphi$ is called a linearisation of the map $G$.
For instance, in \cite{MR1309126} the authors present modern description of the Theorem of {K}er\'{e}kj\'{a}rt\'{o} which justifies the existence of a global linearisation of a  $m-$periodic map (i. e. $F^m=\mbox{Id}$, for some positive integer $m>0$).
This is also studied in the recent paper \cite{MR3991373}, where the authors study the existence of a smooth linearisation of planar periodic maps.
\par
It is proved that \textit{a $C^1-$unipotent map $G$ is globally equivalent to the linear translation, as long as $G$ is fixed point free}. This means that, its fixed point set $\mbox{Fix}(G)=\{(x,y)\in\mathbb{R}^2\colon G(x,y)=(x,y)\}$ is empty, as in \cite{MR1176619}.
This result is obtained after a global characterization of each fixed point free unipotent map, by using a linear conjugacy induced by a rotation (\mbox{Proposition \ref{prop:IIRot2}}).
In the complementary case, $\mbox{Fix}(G)\neq\emptyset$, it is demonstrated not only that the fixed point set induced by a $C^1-$unipotent has no isolated elements, but that \textit{a $C^1-$unipotent map has no periodic points}.
Specifically, these maps are free, in the sense of the paper \cite{MR837985}.
In the proofs, the existent normal form of $C^1-$unipotent maps is important.
For instance, it is utilised in the last section in order to describe a global bifurcation on unipotent maps, related with the existence of global attractor fixed points.
\smallskip
\par
This paper is organized as follows.
\mbox{Section \ref{sec:c1normalform}} presents  the  notably \textit{normal form} for each $C^1-$unipotent two--dimensional map, as proved in the paper \cite{MR1800694}; this important characterization on these bijective maps is used in the other sections of the paper, in order to present a global description of the  continuously differentiable unipotent maps.
\mbox{Section \ref{sec:3}} describes the unipotent maps with a fixed point.
 \mbox{Theorem \ref{thm:normaRotl}} implies  that for each nonlinear unipotent map with a fixed point at the origin, there exists a linear rotation
$R_{\theta}=
\left(
  \begin{array}{cr}
    \cos(\theta) & -\sin(\theta) \\
    \sin(\theta) & \,\cos(\theta)\\
  \end{array}
\right)$
such that
$\big( R_{\theta}  \circ G\circ R_{-\theta}\big)(x,y)=\big(x+\psi(y),y\big)$,
where $\psi$ is a nonlinear function with a fixed point located at zero.
\mbox{Section \ref{sec:4}} includes the definition of a topological disk in $\mathbb{R}^2$ as the subsets $D$ of $\mathbb{R}^2$ that are homeomorphic to the compact set $\overline{D}_1=\{z\in ||z||\leqslant 1\}$.
In this context, \mbox{Theorem \ref{thm:freeDinamics}} proves that a nonlinear unipotent map induces a simple dynamics, in the sense that
\[
G(D)\cap D=\emptyset\quad \Rightarrow \quad G^p(D)\cap G^q(D)=\emptyset,
\]
for each disk topological  $D$ in $\mathbb{R}^2$ and integers $p\neq q$.
As usual, the symbol  $G^m$ denotes the composition $G\circ\cdots \circ G$ (respectively $G^{-1}\circ\cdots \circ G^{-1}$)
$m>0$ (respectively $m<0$) times, and $G^0$ is the identity map.
Consequently, a unipotent map has no periodic points as long as it is continuously differentiable.
\mbox{Section \ref{sec:5}} characterizes the fixed point free unipotent maps, and includes the proof that a  $C^1-$unipotent map  has no isolated fixed points.
\mbox{Section \ref{sec:6}} presents the proof that \textit{a fixed point free unipotent map  is globally equivalent to the  translation  $\tau(x,y)=(x+1,y)$ as long as it is continuously differentiable.}
\mbox{Section \ref{sec:families}} presents one parameter families whose fixed points change its stability, between to be  a global repellor or to be a global attractor.
This is related with \emph{discrete Markus-Yamabe question} \cite{MR1643454,MR1790619}; where the authors study the existence of spectral conditions on Euclidean maps $\mathbb{R}^n\mapsto\mathbb{R}^n$, in order to obtain a global attractor fixed point.

\section{Normal form of two-dimensional unipotent maps}\label{sec:c1normalform} %
It is  accepted  that, under different  circumstances, the unipotent maps in dimension two have been classified and also globally described.
For instance, in the recent  paper \cite{MR1791178}, Chamberland proves that a real--analytic map of the Euclidean  plane into itself has an inverse as long as it is unipotent (see also
\cite{MR1485463,MR2609212}).
An interesting  proof for $C^1-$maps appears in the contemporaneous paper \cite{MR1800694}, where Campbell presents an important normal form of unipotent maps, and also observes that it has an explicit global inverse.
This impressively result  may be enunciated  as follows.
\begin{theorem}[Campbell]\label{teo:campbell}
Let  $G:\mathbb{R}^2\to\mathbb{R}^2$  be  $C^1$.
Then $G$ is unipotent if $G$ is of the form
    \begin{equation}\label{eq:normaC1}
    G(x,y)=\big(x+b\phi(ax+by)+c,y-a\phi(ax+by)+d\big),
    \end{equation}
for some constants $a,b,c,d\in\mathbb{R}$ and some function $\phi$ of a single variable.
If that is the case, then $G$ has an explicit global inverse.
Conversely, if $G$ is $C^1$ and unipotent, then $G$ is of the form above for a $\phi$ that is $C^1$.
\end{theorem}
This normal form has been analysed in \cite{MR1994859}, in the case of maps whose Jacobian matrices have equal eigenvalues, not necessarily one.
Unfortunately, this normal form does not exist outside the unipotent maps as shown the interesting real--analytic map,  presented in \cite[Theorem 1.3]{MR1994859}.

\section{Unipotent maps with a fixed point}\label{sec:3}
In this context, a $C^1-$map $G:\mathbb{R}^2\to\mathbb{R}^2$  is unipotent, if and only if, it has the form \eqref{eq:normaC1}. where $a,b,c$ and $d$ are real constants, and $\phi$ is a $C^1-$function on a single variable.
Consequently,
\begin{itemize}
  \item $\phi(0)=0$ implies that $G(0,0)=(c,d)$.
  \item $(0,0)=(a,b)$ means that $G(x,y)=(x+c,y+d)$, it is a linear map.
\end{itemize}
\begin{proposition}\label{prop:normal}
Let $G\colon\mathbb{R}^2\to\mathbb{R}^2$ be a $C^1-$map.
Then the following statements are equivalent
\begin{enumerate}
  \item\label{prop:normal(1)}  $G$ is a non-linear unipotent map such that $G(0,0)=(0,0)$.
  \item \label{prop:normal(2)} There are  $\phi$, a non-linear $C^1-$function of a single variable joint to $(b,-a)$, a non-zero constant vector such that
          \[
    G(x,y)=\Big(x+b\phi\big(ax+by\big)-b\phi(0),y-a\phi\big(ax+by\big)+a\phi(0)\Big).
          \]
  \item \label{prop:normal(3)} There are $\psi$, a non-linear $C^1-$function of a single variable with $\psi(0)=0$ joint to $(\beta,-\alpha)$, a unitary constant vector such that
            \[
    G(x,y)=\Big(x+\beta\psi\big(\alpha x+\beta y\big),y-\alpha\psi\big(\alpha x+\beta y\big)\Big).
    \]
\end{enumerate}
\end{proposition}
\begin{proof}
To demostrate  that \eqref{prop:normal(1)} implies \eqref{prop:normal(2)}, let  $G\colon\mathbb{R}^2\to\mathbb{R}^2$  refer the $C^1-$map  in the first statement.
The existence of the non-linear map $\phi $ and the non-zero vector $(b,-a)$, described in \eqref{prop:normal(2)},  is  a consequense of the normal form \eqref{eq:normaC1}.
In addition,  $G(0,0)=(0,0)$ implies that $0=c+b\phi(0)$ and $0=d-a\phi(0)$. Therefore, announcement \eqref{prop:normal(2)}  holds.
\par
In order obtain that \eqref{prop:normal(2)} implies  \eqref{prop:normal(3)}, the initial observation is that  \eqref{prop:normal(2)} gives the existence of  the  unitary vector $(\beta,-\alpha)\in\mathbb{R}^2$, defined by
$$\alpha=\dfrac{a}{\sqrt{a^2+b^2}}\quad \mbox{and} \quad \beta=\dfrac{b}{\sqrt{a^2+b^2}}.$$
Similarly, $\phi$ induces the map given by
$$\psi(t)=\sqrt{a^2+b^2}\Big({\phi\big(t\sqrt{a^2+b^2}\big)-\phi(0)\Big)}.$$
This map $\psi$ satisfies $\psi(0)=0$ and  $\psi(\alpha x+\beta y)={\sqrt{a^2+b^2}}\Big({\phi(ax+by)}-\phi(0)\Big)$.
Consequently,
\begin{eqnarray*}
\beta\psi\big(\alpha x+\beta y\big)&=&b\phi\big(a x- b y\big)-b\phi(0),\\
-\alpha{\psi}\big(\alpha x+\beta y\big)&=&-a\phi\big(a x-b y\big)+a\phi(0).
\end{eqnarray*}
Therefore,   \eqref{prop:normal(2)} implies  \eqref{prop:normal(3)}.
\par
Finally, as the non-linear map $\psi$ satisfies $\psi(0)=0$, the  map $G$ is non-linear and satisfies $G(0,0)=(0,0)$.
In addition, a  direct computation shows that   $G(x,y)=\Big(x+\beta\psi\big(\alpha x+\beta y\big),y-\alpha\psi\big(\alpha x+\beta y\big)\Big)$ is unipotent, and consequently  \eqref{prop:normal(3)} implies \eqref{prop:normal(1)}.
Therefore,  this proposition holds.
\end{proof}

%
For each $\theta\in\mathbb{R}$, let  $R_{\theta}$ denote the linear rotation.
\[
R_{\theta}=
\begin{bmatrix}
    \cos(\theta) & -\sin(\theta) \\
    \sin(\theta) & \,\cos(\theta)
\end{bmatrix}.
\]
\begin{theorem}\label{thm:normaRotl}
Let $G\colon\mathbb{R}^2\to\mathbb{R}^2$ be a $C^1-$ map.
Then the following  are equivalent
\begin{enumerate}
  \item\label{prop:normalRot(1)}  $G$ is a non-linear unipotent map such that $G(0,0)=(0,0)$.
\item  \label{prop:normalRot(2)} There is a rotation $R_{\theta}$ such that
            \[
 \big( R_{\theta}  \circ G\circ R_{-\theta}\big)(x,y)=\big(x+\psi(y),y\big),
    \]
where $\psi$ is a non-linear $C^1-$function such that $\psi(0)=0$.
\end{enumerate}
\end{theorem}
\begin{proof}
In satement  \eqref{prop:normalRot(1)} conditions, \mbox{Proposition \ref{prop:normal}}  establishes   the existence of $\psi$, a non-linear $C^1-$function with
 $\psi(0)=0$ joint to $(\beta,-\alpha)$, a unitary constant vector such that
 \[
 G(u,v)=\Big(u+\beta\psi\big(\alpha u+\beta v\big),v-\alpha\psi\big(\alpha u+\beta v\big)\Big).
 \]
In this situation, the well defined rotation
\[
R_{\theta}=
\begin{bmatrix}
    \beta & -\alpha \\
    \alpha & \beta
\end{bmatrix}.
\]
not only  sends the unitary vector $(\beta,-\alpha)$ into the vector $(1,0)$, but $(\alpha,\beta)$ into  $(0,1)$.
In addition, the inverse $R_{-\theta}$ satisfies $R_{-\theta}(x,y)=(\beta x+ \alpha y,-\alpha x+\beta y)$ and then
$ \big( G\circ R_{-\theta}\big)(x,y)=\big(\beta x+ \alpha y+\beta\psi(y),-\alpha x+\beta y-\alpha\psi(y)\big).$
To be precise,
\[
 \big( G\circ R_{-\theta}\big)(x,y)=\big(\beta x+ \alpha y,-\alpha x+\beta y\big)+\psi(y)\big(\beta,-\alpha\big).
\]
Therefore, $ \big( R_{\theta}  \circ G\circ R_{-\theta}\big)$ satisfies  \eqref{prop:normalRot(2)}.
It concludes that  \eqref{prop:normalRot(1)}  implies \eqref{prop:normalRot(2)}.
\par
Finally, \eqref{prop:normalRot(2)} shows that the Jacobian matrix $DG$ has ta form
\[
R_{-\theta}
\circ
\begin{bmatrix}
    1 & * \\
    0 & 1
\end{bmatrix}.
\circ
R_{\theta}.
\]
Thus, $G$ is unipotent and \eqref{prop:normalRot(1)} is true. Therefore this proposition holds.
\end{proof}
\begin{remark}
In the case that $G$ is a  linear unipotent map, the stardard Jordan Form Theory direcly gives the existence of a linear isomorphism
$T\colon\mathbb{R}^2\to\mathbb{R}^2$ such that $ \big( T  \circ G\circ T^{-1}\big)(x,y)=\big(x+B y,y\big)$, where $B\in\mathbb{R}^2$ is a constant.
\end{remark}

\section{Dynamics of unipotent maps with a fixed point located at the origin} \label{sec:4}
The characterization presented \mbox{Theorem \ref{thm:normaRotl}} is used in this section in order to describe the dynamical properties of nonlinear unipotent maps with a fixed point.
It is motivated by the important description presented in \cite{MR837985}, where the author introduce a new class of homeomorphisms, called \emph{free homeomorphisms} (\mbox{Theorem \ref{thm:freeDinamics}}).

\begin{theorem}\label{thm:trivialOnSegments}
Let $G\colon\mathbb{R}^2\to\mathbb{R}^2$ be a non-linear unipotent $C^1-$map also with  $G(0,0)=(0,0)$.
That is, there is a rotation $R_{\theta}$ such that
$ \big( R_{\theta}  \circ G\circ R_{-\theta}\big)(x,y)=\big(x+\psi(y),y\big)$,
where $\psi$ is a non-linear $C^1-$function  satisfying  $\psi(0)=0$.
Then,  for each open interval $I\subset\{y\in\mathbb{R}\colon\psi(y)\neq0\}$, the vertical  segment $\Delta_x=\{x\}\times I$ satisfies
$$ G^{-1}_{\theta}(\Delta_x)\cap  \Delta_x =\emptyset \quad \mbox{ and }\quad  \Delta_x \cap G_{\theta}(\Delta_x)=\emptyset,$$
where $G_{\theta}= R_{\theta}  \circ G\circ R_{-\theta}$ with  $ G^{-1}_{\theta}$ its  inverse, and $x\in\mathbb{R}$.
In addition, when such an  interval $I$ is maximal (a connected component), then for every endpoint, say  $a\in\mathbb{R}$, the  image satisfies
$$G_{\theta}(x,a)=(x,a),\quad\forall x \in\mathbb{R}.$$
\end{theorem}

\begin{proof}
In the open interval, the map does not change its sign.
So without loos of generality, the interval has the following form $I\subset\{y\in\mathbb{R}\colon\psi(y)>0\}$.
Thus, if  $(x,y)\in\Delta_x$ the image $G_{\theta}(x,y)$ and its inverse $G^{-1}_{\theta}(x,y)$ satisfy.
$$G_{\theta}(x,y)-(x,y)=(\psi(y),0)=(x,y)-G^{-1}_{\theta}(x,y).$$
These differences are different from zero.
Therefore,  the first part of the theorem  holds.
\par
To conclude, let $a\in\mathbb{R}$ be an endpoint of a maximal interval $I$, given by $\{y\in\mathbb{R}\colon\psi(y)>0\}$.
As the continuos function $\psi$  is defined in the whole $\mathbb{R}$, and $I$ maximal, the value $\psi(a)=0$, consequently
$G_{\theta}(x,a)=(x,a)$. This complete the proof.
\end{proof}
\begin{example}
The illustrative unipotent maps
$$(x,y)\mapsto(x+y^3,y)\quad \mbox{ and }\quad(x,y)\mapsto(x+y^2,y)$$
have two different behaviors around the horizontal axis, where both maps have all their fixed points.
\end{example}
The next corollary considers  the notations of \mbox{Theorem \ref{thm:trivialOnSegments}}
Furthermore,  for each posiitive integer  $m>0$, the symbol  $G^m$ denotes the composition $G\circ\cdots \circ G$, $m$ times ($G^0$ is the identity map).
\begin{corollary}
Let $G\colon\mathbb{R}^2\to\mathbb{R}^2$ be a map, as in \mbox{Theorem \ref{thm:trivialOnSegments}}.
If  $I\subset\mathbb{R}$ is  the respective maximal interval.
Then for any $z\in R^{-1}_{\theta}\big(\mathbb{R}\times I \big)$ the sequence $\{G^n(z)\colon n\geq 0\}$, induced by
 compositions is well defined and it is divergent, in the sense that
$$\displaystyle\lim_{n\to+\infty}||G^n(z)||=+\infty.$$
Otherwise, $G(\tilde{z})=\tilde{z}$.
\end{corollary}
\begin{proof}
If  $z\in R^{-1}_{\theta}\big(\mathbb{R}\times I \big)$ the point $(x,y)=R_{\theta}(z)$ satisfies that
$$G_{\theta}^n(x,y)=(x+n\psi(y),y)\quad\forall n\geq 0.$$
This sequence  is unbounded, because $y\in I$, where the value $\psi(y)\neq 0$.
Finally,  it is enough to observe that
$$G^{n}=R_{-\theta}\circ G_{\theta}^n\circ R_{\theta}$$
Therefore, the first part of the corollary holds.
\par
The second part follows, since the complement set of
$$\bigcup \Big\{R^{-1}_{\theta}\big(\mathbb{R}\times I \big)\colon I  \mbox{ is maximal }\Big\},$$
is contained in $Fix(G)=\{p\in\mathbb{R}^2\colon G(p)=p\}$, as shown in the last part of  \mbox{Theorem \ref{thm:trivialOnSegments}}.
This concludes the proof.
\end{proof}
\begin{corollary}
Let $G\colon\mathbb{R}^2\to\mathbb{R}^2$ be a map, as in \mbox{Theorem \ref{thm:trivialOnSegments}}.
If  $I\subset\mathbb{R}$ is  the respective maximal interval.
Then for any $w, z\in R^{-1}_{\theta}\big(\mathbb{R}\times I \big)$  there exists a compact segment $\Lambda \subset R^{-1}_{\theta}\big(\mathbb{R}\times I \big)$ whose endpoints are exactly $z$ and $w$ such that
$$\displaystyle\lim_{n\to+\infty}G^n(\Lambda)=\infty\quad \mbox{ and } \quad \displaystyle\lim_{m\to+\infty}G^{-m}(\Lambda)=\infty.$$
It means that for every compact set $K\subset\mathbb{R}^2$ there is a natural number $\tilde{n}\in\mathbb{N}$ such that $G^n(\Lambda)\cap K=\emptyset$ when $|n|>\tilde{n}$.
\end{corollary}
\begin{proof}
The points  $w, z\in R^{-1}_{\theta}\big(\mathbb{R}\times I \big)$ means that $ R_{\theta}(z)$ and $ R_{\theta}(w)$ belong to the band
 $\mathbb{R}\times I$.
Thus, the compact segment $\{t R_{\theta}(z)+(1-t)R_{\theta}(w)\colon0\leq t \leq 1\}$ is not only contained in
$\mathbb{R}\times I$, but its image
$$\Lambda= R^{-1}_{\theta}\big(\{t R_{\theta}(z)+(1-t)R_{\theta}(w)\colon0\leq t \leq 1\}\big)$$
is the compact segment whose endpoints are $z$ and $w$.
This $\Lambda$ induces, by a projection, the compact interval $\{y\in\mathbb{R}\colon R_{-\theta}(x,y) \in\Lambda \}$ where the restriction of $\psi$ has its maximum and minimum, both different from zero, and with the same sign.
Therefore, this $\Lambda$ satisfies the conditions of the corollary.
\end{proof}

\begin{remark}
Under the notations of  \mbox{Theorem \ref{thm:trivialOnSegments}}, the connected components of the open set
$$\bigcup \Big\{R^{-1}_{\theta}\big(\mathbb{R}\times I \big)\colon I  \mbox{ is maximal }\Big\}\neq\emptyset$$
are \emph{fundamental regions}, in the sense of \cite{MR172258}.
These fundamental regions are invariant sets, where the inclusion of a simple closed curve (Jordan Curve) implies the inclusion of the open set enclosed by it.
\end{remark}

In the next theorem, a  subset $D$ of $\mathbb{R}^2$ is called a topological disc in $\mathbb{R}^2$ when it is homeomorphic to $\overline{D}_1=\{z\in ||z||\leq 1\}$.

\begin{theorem}\label{thm:freeDinamics}
Let $G\colon\mathbb{R}^2\to\mathbb{R}^2$ be a  unipotent $C^1-$map, with  $G(0,0)=(0,0)$.
Then
$$G(D)\cap D=\emptyset\implies G^p(D)\cap G^q(D)=\emptyset,$$
for each disk topological  $D$ in $\mathbb{R}^2$ and integers $p\neq q$.
\end{theorem}

\begin{proof}
In the linear case, this map becomes the identity map.
In the non-linear case, \mbox{Theorem \ref{thm:normaRotl}}  shows that there is no loss of generality by writing:
\[
G(x,y)=\big(x+\psi(y),y\big),
\]
for some $C^1-$function  $\psi$,  with $\psi(0)=0$.
In this context, the fixed point set $\{z\in\mathbb{R}^2\colon G(z)=z\}$ are the  horizontal lines $y=\gamma$, where $\psi(\gamma)= 0$.
In a different  situation, that is $G(x,y)\neq (x,y)$, the sequence
\[
G^m(x,y)=(x+m\psi(y),y), \quad \forall m\in\mathbb{Z}.
\]
Under these conditions, the assumption $G(D)\cap D=\emptyset$ implies that the closed set $D$ is contained in an open band of the form
$$\mathbb{R}\times I \quad\mbox{where}\quad I\subset\{y\in\mathbb{R}\colon \psi(y)\neq0\}.$$
Consequently
$$G(D)\cap D=\emptyset\implies G^m(D)\cap G(D)=\emptyset,\quad\forall m>0.$$
Thus, the result  is obtained by using  either $m=p-q$ or $m=q-p$, the positive one.
Therefore, this theorem holds.
\end{proof}

\begin{remark}
In the terminology of  \cite{MR837985}, the maps in \mbox{Theorem \ref{thm:freeDinamics}}  are \emph{free}.
Consequently,  the conclusions of \mbox{Theorem \ref{thm:freeDinamics}} remain correct with $D$ replaced by  a continuum, this is a
 compact and connected subset of $\mathbb{R}^2$.
For instance,  each one point set $\{z\}\subset\mathbb{R}^2$ is a  continuum.
\end{remark}
\begin{theorem}\label{thm:5.10}
Let $G\colon\mathbb{R}^2\to\mathbb{R}^2$ be a  unipotent $C^1-$map, with  $G(0,0)=(0,0)$.
When  $z\in\mathbb{R}^{2}$, and $G(z)\neq z$, then the exist a line $\ell_z\subset\mathbb{R}^{2}$ such that
$$ G^{-1}(\ell_z)\cap \ell_z =\emptyset \quad \mbox{ and }\quad  \ell_z \cap G(\ell_z)=\emptyset,$$
but $z\in\ell_z$.
\end{theorem}
\begin{proof}
As in the proof of \mbox{Theorem \ref{thm:freeDinamics}}, there is no loss of generality by assuming that:
$G(x,y)=\big(x+\psi(y),y\big)$, for some $C^1-$function  $\psi$,  with $\psi(0)=0$.
Thus, $\ell_z$  corresponds to the vertical line passing through $z\in\mathbb{R}^{2}$.
Therefore, this theorem holds.
\end{proof}

A an element $(x,y)\in\mathbb{R}^2$ is called \textbf{periodic point} of $G$ provided the existence of some integer $p>1$ such that
\[
G^p(x,y)=(x,y)\quad\mbox{ but } \quad G^m(x,y)=(x,y),\quad\forall 1\leq m \leq p-1.
\]
Notice that, this periodic point is also a fixed point of $G^p$.
\begin{theorem}\label{thm:noPeriodic}
A $C^1-$unipotent map $G\colon\mathbb{R}^2\to\mathbb{R}^2$ has no periodic points
\end{theorem}

\begin{proof}
There is no loss of generality by assuming that:
$G(x,y)=\big(x+\psi(y),y\big)$, for some $C^1-$function  $\psi$,  with $\psi(0)=0$.
Consequently, the condition $G(x,y)\neq(x,y)$, that meas $\psi(y)\neq0$, directly  implies that
\[
G^m(x,y)-(x,y)=(m\psi(y),0)\neq0, \quad \forall m\in\mathbb{Z}\setminus\{0\}.
\]
Therefore, $G$ has no fixed points.
\end{proof}

It should be mentioned that the results not only give a description of the full dynamics, but it presented a smooth conjugacy
of the system with systems of the form $(x,y)\mapsto\big(x+\psi(y),y\big)$, where, clearly, the degree of the map $(x,y)\mapsto\big(\psi(y),0\big)$ is different from one.


\section{Characterization of fixed point free unipotent maps}\label{sec:5}
In order to present a complete characterization of the unipotent maps, they are described in two different cases.
The existence of the normal form in  \mbox{Theorem \ref{teo:campbell}} remains correct with $\phi$ exchanged by $t\mapsto \phi(t)-\phi(0)$.
Therefore, it is not difficult to see that  a $C^1-$map $G:\mathbb{R}^2\to\mathbb{R}^2$  is unipotent, if and only if, it has the form
\begin{equation}\label{eq:SecondNormaC1}
    G(x,y)=\big(x+b\phi(ax+by)+c,y-a\phi(ax+by)+d\big),
\end{equation}
where $a,b,c$ and $d$ are real constants, and $\phi$ is a $C^1-$function such that $\phi$ sends zero into zero.
\par
In this context,  there is no ambiguity in the presentation of the following sets.
Specifically,
\[
\mathcal{UP}_1=\Big\{G:\mathbb{R}^2\to\mathbb{R}^2 \colon \mbox{ In } \eqref{eq:SecondNormaC1},  (c,d)=(0,0) \mbox{ and } \phi(0)=0 \},
\]
and
\[
\mathcal{UP}_2=\Big\{G:\mathbb{R}^2\to\mathbb{R}^2 \colon \mbox{ In } \eqref{eq:SecondNormaC1},  (c,d)\neq(0,0) \mbox{ and } \phi(0)=0 \}.
\]
Notice that, the  union $\mathcal{UP}_1\cup\,\mathcal{UP}_2$ coincides with the  set of all the $C^1-$unipotent maps of $\mathbb{R}^2$ into itself.
In addition,
\begin{equation}\label{eq:IRot1}
G\in\mathcal{UP}_1\implies G(0,0)=(0,0).
\end{equation}
Therefore, $\mathcal{UP}_1$ has  no fixed point free maps.
\begin{proposition}\label{prop:IIRot2}
Let $G\colon\mathbb{R}^2\to\mathbb{R}^2$ be $C^1-$unipotent map of the form
    \[
    G(x,y)=\big(x+b\phi(ax+by)+c,y-a\phi(ax+by)+d\big)\quad \forall  x,y \in\mathbb{R},
    \]
    where  $a,b,c,d\in\mathbb{R}$ and $\phi$ is a $C^1-$function such that $\phi(0)=0$.
Then the following are equivalent.
\begin{enumerate}
  \item\label{prop:IIRot(1)} The non--linear map $G\in\mathcal{UP}_2$ is  fixed point free.
\item  \label{prop:IIRot(2)} There is a rotation $R_{\theta}$ such that
            \[
 \big( R_{\theta}  \circ G\circ R_{-\theta}\big)(x,y)=\big(x+\psi(y),y\big)+R_{\theta}(c,d),
    \]
where $R_{\theta}(c,d)$ is a non--zero constant vector, and   $\psi$ is a $C^1-$function with $\psi(0)=0$ such that
$$
\begin{bmatrix}
    \psi(t) \sqrt{a^2+b^2}+(bc-ad) \\
   ac+bd
\end{bmatrix}
\neq
\begin{bmatrix}
   0 \\
 0
\end{bmatrix}
\quad \forall t\in\mathbb{R}.
$$
\end{enumerate}
\end{proposition}
\begin{proof}
The map $G\in\mathcal{UP}_2$  is non-linear.
Consequently,   the  constants $\alpha=\frac{a}{\sqrt{a^2+b^2}}$,  $ \beta=\frac{b}{\sqrt{a^2+b^2}}$  and the map
$\psi_2(t)=\sqrt{a^2+b^2}\phi(t\sqrt{a^2+b^2})$ are well defined.
Under these conditions,
$$ G(u,v)=\big(u+\beta\psi_2(\alpha u+\beta v)+c,v-\alpha\psi_2(\alpha u+\beta v)+d\big)$$
with $\alpha^2+\beta^2=1$ ,  $\psi_2(0)=0$ and $(c,d)\neq(0,0)$.
\par
The proposition is obtained  the rotation
$R_{\theta}=
\begin{bmatrix}
    \beta & -\alpha \\
    \alpha & \beta
\end{bmatrix}$
whose inverse can be write as $R_{-\theta}(x,y)=(\beta x+ \alpha y,-\alpha x+\beta y)$.
Consequently, $ \big(G\circ R_{-\theta}\big)(x,y)$ is equal to $(\beta x+\alpha y,\beta y-\alpha x)+(\beta\psi_2(y),-\alpha\psi_2(y))+(c,d)$ and then
\[
 \big( R_{\theta}  \circ G\circ R_{-\theta}\big)(x,y)=\big(x+\psi(y),y\big)+(c\beta-\alpha d,\alpha c+\beta d).
 \]
Therefore, \eqref{prop:IIRot(2)} is true.
\par

The reverse conclusion is obtained by a direct computaiton and, therefore, this proposition holds.
\end{proof}

\begin{remark}\label{rem:5.15}
The linear unipotent maps in  $\mathcal{UP}_2$  have cases with  a similar form in \mbox{Theorem \ref{prop:IIRot2}}.
In this situation, when  $G\in\mathcal{UP}_2$, there exist a linear isomorphism $T\colon\mathbb{R}^2\to\mathbb{R}^2$ such that
\begin{equation}\label{eq:jordanLinear}
\big(T\circ G \circ T^{-1}\big)(u,v)=(u+Bv,v)+(C,D),
\end{equation}
where $B\in\mathbb{R}$ and $(C,D)\neq(0,0)$.
In addition, the rotation
\[
R_{\Theta}=\frac{1}{C^2+D^2}
\begin{bmatrix}
    C & D \\
    -D & C
\end{bmatrix}
\]
satisfies
\[
\big( R_{\Theta} \circ G \circ R_{-\Theta} \big)(x,y)=(x+1,y)+
\dfrac{B}{C^2+D^2}
\begin{bmatrix}
    CD & C^2 \\
    -D^2 & -DC
\end{bmatrix}
\begin{bmatrix}
    x \\
    y
\end{bmatrix}.
\]
\par
The particular condition, $D=0$ implies
$$\big( R_{\Theta} \circ G \circ R_{-\Theta} \big)(x,y)=(x+1,y),$$
because $B\neq0$ induces a fixed point as long as $D=0$.
\par
The general case $D\neq0$ might be  studied in \eqref{eq:jordanLinear} with $B\neq0$.
In this case, the projection into the vertical axes is different form zero.
Thus $\mathbb{R}^{2}$ is  a fundamental region and concludes that \eqref{eq:jordanLinear} is conjugated to a linear translation $(x,y)\mapsto(x+1,y)$ \cite{MR172258}.
\end{remark}

\begin{theorem}
A $C^1-$unipotent map $G\colon\mathbb{R}^2\to\mathbb{R}^2 $ has no isolated fixed points.
\end{theorem}
\begin{proof}
In the linear case, the map  becames the identity map, as long as  it has fixed points.
In the non-linear case, $G$ satifies  {Proposition \ref{prop:normal}}.
So the fixed points appear in a line  of the form $\alpha x+\beta y=0$.
Thus, this theorem holds.
\end{proof}
\section{Dynamics of fixed point free $C^1-$unipotent maps} \label{sec:6}
In the global description of $C^1-$unipotent maps, the initial observation  follows from the proof of  \mbox{Theorem \ref{thm:normaRotl}}.
To be precise, any  $C^1-$unipotent map admits a decomposition  of $\mathbb{R}^2$ in a family of  parallel lines such that the map preserves each such a  line, and sends it  homeomorphically  into itself.
\begin{proposition}\label{prop:5.2}
Let  $G\in\mathcal{UP}_2$ be a fixed point free map.
Then, for any pair of points $z,w\in\mathbb{R}^2$ the compact connected segment $\Lambda= \{t z+(1-t)w\colon0\leq t \leq 1\}$ satisfies
$$\displaystyle\lim_{n\to+\infty}G^n(\Lambda)=\infty\quad \mbox{ and } \quad \displaystyle\lim_{m\to+\infty}G^{-m}(\Lambda)=\infty.$$
\end{proposition}

\begin{proof}
The linear case has been analyzed in \mbox{Remark \ref{rem:5.15}}.
In the non-linear case, \mbox{Proposition \ref{prop:IIRot2}}  shows that there is no loss of generality by writing:
\[
G(x,y)=\big(x+\psi(y)+C,y+D\big),
\]
for some $C^1-$function  $\psi$ such that  $\psi(0)=0$, and  $(\psi(y)+C,D)\neq(0,0)$,  for all $y\in\mathbb{R}$.
If $D=0$, the condition $\psi(y)+C\neq0$ implies that the compact conected segment
$\Lambda= \{t z+(1-t)w\colon0\leq t \leq 1\}$ satisfies the requested conditions.
Therefore, in this first case the proposition holds.
\par
If $D\neq0$, the projection of $G(x,y)-(x,y)$ into the vertical axis is different from zero.
Consequently, its sign in the connected set $\{y\colon(x,y\in\Lambda)\}$ is constant.
Thus, $\Lambda$ satisfies the limits.
Therefore, this propositions holds.
\end{proof}

As usual, the homeomorphisms $F$ and $G$ of $\mathbb{R}^{2}$ into itself, are \textbf{conjugated} if there exists an homeomorphisms $\varphi\colon\mathbb{R}^{2}\to\mathbb{R}^{2}$ such that $\varphi\circ F= G\circ \varphi$.

\begin{theorem}
If the $C^1-$unipotent map $G\colon\mathbb{R}^2\to\mathbb{R}^2 $ is fixed point free, then $G$ is topologically conjugated to the  global translation  $\tau(x,y)=(x+1,y)$.
\end{theorem}

\begin{proof}
This theorem is proved by using \cite{MR172258}.
In this paper, the author studies the orientation preserving homeomorphisms $G\colon\mathbb{R}^2\to\mathbb{R}^2 $ whose $Fix(G)=\emptyset$, and demostrates that $G$ is topologically conjugated to the  global translation  $\tau(x,y)=(x+1,y)$ if the whole plane is a \emph{fundamental region}.
Thus, this theorem follows by \mbox{Proposition \ref{prop:5.2}}, where is proved that $\mathbb{R}^{2}$ is a \emph{fundamental region}, in the sense of \cite{MR172258}.
Therefore, the theorem holds.
\end{proof}
\section{Families where the fixed point changes its stability}\label{sec:families}%
This section, motivated by \cite{MR909943}, is concerned with the description of the  simplest patterns according to which unipotent maps
of the form $G_{\mu}(x,y)=G(x,y)-(\mu x,\mu y)$, where $G\colon\mathbb{R}^2\to\mathbb{R}^2$ is a non--linear unipotent map with $G(0,0)=(0,0)$, change its stability -- bifurcate -- under perturbations of the parameter $\mu$,  in a small open interval centred at zero.
\par
The next definitions, presented in \cite{MR2417859}, will be needed.
\begin{definition}
Let $F\colon\mathbb{R}^2\to\mathbb{R}^2 $ be a topological embedding; that is, a globally injective local homeomorphism.
\begin{itemize}
  \item Let $p\in\mathbb{R}^2$. The \textbf{$\omega-$limit set of $p$ }is
    \[
    \omega(p)=\Big\{z\in\mathbb{R}^{2}\colon \exists,  0< n_k\in\mathbb{N},
    \,\mbox{ such that }\, \lim_{n_k\to\infty}F^{n_k}(p)=z\Big\}.
    \]
    \item The origin $(0,0)$ is a \textbf{local attractor} \big(resp. \textbf{local repellor}\big) for $F$ if there exist a topological disc $D$, which is contained in the domain of definition of $F$ (resp. $F^{-1}$), that is a neighbourhood of $(0,0)$ such that $F(D)\subset \mbox{Int}(D)$ \big(resp. $F^{-1}(D)\subset \mbox{Int}(D)$\big) and $\displaystyle\cap_{n=1}^{\infty}F^{n}(D)=\{(0,0)\}$ \big(resp. $\displaystyle\cap_{n=1}^{\infty}F^{-n}(D)=\{(0,0)\}$\big).
  \item The origin $(0,0)$ is a \textbf{global attractor} for $F$ if $(0,0)$ is a local attractor  for $F$ and $\omega(p)=\{(0,0)\}$ for all $p\in\mathbb{R}^2$.
\end{itemize}
\end{definition}

\begin{lemma}\label{lem:attractor}
Let $F\colon\mathbb{R}^2\to\mathbb{R}^2 $ be an orientation preserving $C^1-$embedding.
Assume that $F(x,y)=\big(\lambda x+\psi(y),\lambda y\big)$, where the constant $0<\lambda<1$ and the function $\psi$ has a fixed point at zero.
Then $(0,0)$ is a global attractor for $F$.
\end{lemma}

\begin{proof}
The origin $(0,0)$ is an  hyperbolic attractor. In addition,
\[
F^n(x,y)=\left(\lambda^n x+\sum_{k=0}^{n-1}\lambda^{n-1-k}\psi\big(y\lambda^{k} \big),\lambda^n y\right),\quad \forall n\geqslant1.
\]
In this context,
\[
\left|\sum_{k=0}^{n-1}\lambda^{n-1-k}\psi\big(y\lambda^{k} \big)\right|
\leqslant
\max_{0\leqslant k\leqslant n-1}|\psi\big(y\lambda^{k} \big)|\frac{1}{1-\lambda}.
\]
Thus, $\displaystyle\lim_{n\to\infty}F^n(x,y)=(0,0)$, because $\psi(0)=0$, and then $\omega(x,y)=\{(0,0)\}$ for all $(x,y)\in\mathbb{R}^2$.
Therefore, $(0,0)$ is a global attractor for $F$.
\end{proof}

\begin{definition}\label{def:repellor}
Let $F\colon\mathbb{R}^2\to\mathbb{R}^2$ be a homeomorphism.
The fixed point $(0,0)$ is a \textbf{global  repellor} for $F$ if $(0,0)$ is a global attractor for the inverse $F^{-1}$.
\end{definition}

The next theorem presents a family where the fixed point changes its stability.

\begin{theorem}
Let $G\colon\mathbb{R}^2\to\mathbb{R}^2 $ be a $C^1-$map with $G(0,0)=(0,0)$.
If the map $G$ is non-linear and unipotent, then there exists $\varepsilon>0$ such that the family of maps $$\big\{G_{\mu}(x,y)=G(x,y)-(\mu x,\mu y)\colon-\varepsilon<\mu<\varepsilon\big\}$$
satisfies the following two conditions:
\begin{itemize}
  \item[(a)] For $\mu>0$ the map $G_{\mu}$ has a global  attractor at $(0,0)$, and for $\mu<0$ the map $G_{\mu}$ has a global repellor  at $(0,0)$.
  \item[(b)] The map $G_{\mu}$ has no periodic points in $\mathbb{R}^2$ for $-\varepsilon<\mu<\varepsilon$.
\end{itemize}
\end{theorem}

\begin{proof}
Set $0<\varepsilon\leqslant1$.
\mbox{Theorem \ref{thm:normaRotl}} implies that  the local diffeomorphisms  $G_{\mu}$ are proper, consequently the family only includes diffeomorphisms.
\par
If $\mu>0$, the spectrum $\mbox{Spc}(G_{\mu})=\{1-\mu\}$ satisfies $0<1-\mu<1$, $G_{\mu}(0,0)=(0,0)$.
Thus, \mbox{Theorem \ref{thm:normaRotl}} and \mbox{Lemma \ref{lem:attractor}} imply that $G_{\mu}$ has a global  attractor at $(0,0)$.
If $\mu<0$, \mbox{Theorem \ref{thm:normaRotl}} implies that
\[
G^{-1}_{\mu}(u,v)=\left(\frac{u}{1-\mu}+\psi\Big(\frac{v}{1-\mu}\Big),\frac{v}{1-\mu}\right),
\]
where $0<\frac{1}{1-\mu}<1$ and $\psi(0)=0$.
\mbox{Lemma \ref{lem:attractor}} and \mbox{Definition \ref{def:repellor}} show that  $G_{\mu}$ has a global repellor  at $(0,0)$. Therefore, statement (a) holds.
\par
The item (b) follow by using (a) and \mbox{Theorem \ref{thm:noPeriodic}}.
\end{proof}

\begin{remark}
In \cite{MR2417859} appears a smooth diffeomorphims $F\colon\mathbb{R}^2\to\mathbb{R}^2$ which has an order four periodic point, and is such that $F(0,0)=(0,0)$, its spectrum  $Spc(F)\subset\{z\in\mathbb{C}\colon||z||<1\}$, and $\infty$ is a repellor in the sense that $\infty$ is a local repellor of the natural extension to a homeomorphism  $F\colon\mathbb{R}^2\cup\{\infty\}\to\mathbb{R}^2\cup\{\infty\}$ of the Riemann sphere, with $F(\infty)=\infty$.
\end{remark}

\section*{Acknowledgements}
This paper was written while the author served as an Associate Fellow at the Abdus Salam {\sc ictp} in
Italy; we wish to thank the members of the Mathematics Group for their kind hospitality.
\providecommand{\bysame}{\leavevmode\hbox to3em{\hrulefill}\thinspace}
\providecommand{\MR}{\relax\ifhmode\unskip\space\fi MR }
\providecommand{\MRhref}[2]{%
  \href{http://www.ams.org/mathscinet-getitem?mr=#1}{#2}
}
\providecommand{\href}[2]{#2}

\end{document}